\documentclass{amsart}
\usepackage{amsmath,amssymb}

\newtheorem{Theorem}{Theorem}[section]

\newtheorem{Proposition}[Theorem]{Proposition}

\newtheorem{Definition}[Theorem]{Definition}
\newtheorem{rem}[Theorem]{Remark}
\numberwithin{equation}{section}

\newcommand{\R}{\mathbb R}

\newcommand{\N}{\mathbb N}
\newcommand{\Z}{\mathbb Z}

\newcommand{\F}{\mathcal{F}}
\newcommand{\p}{\mathcal{P}}

\title[On diffeological PDO bundles]{On diffeological principal bundles of non-formal pseudo-differential operators over formal ones}
\author{Jean-Pierre Magnot}
	\address{University of Angers, CNRS, LAREMA, SFR MATHSTIC, F-49000 Angers, France and   Lyc\'ee Jeanne d'Arc,  Avenue de Grande Bretagne,  63000 Clermont-Ferrand, France}

\begin{document}
	\begin{abstract}
		We describe the structure of diffeological bundle of non formal classical pseudo-differential operators over formal ones, and its structure group. For this, we give few results on diffeological principal bundles with (a priori) no local trivialization, use the smoothing connections alrealy exhibited by the author in previous works, and finish with open questions.
	\end{abstract}

		\maketitle
	
	\section{Introduction}	
	The correspondence between formal and non-formal pseudo-differential operators plays a highly important role in analysis. Over the full interest of non-formal pseudodifferential operators for spectral analysis, theory of partial differential equations and differential operators, integration on infinite dimensional spaces, renormalization among other themes, over the crucial place of formal pseudo-differential operators for integrable systems, theory of r-matrices, representation theory, deformation quantization among other fields of studies, it often happens that technical steps can be overcome by passing from a family of non-formal pseudo-differential operators to its formal coounterpart, and by interpreting the results. The reverse procedure also exists, but less often because there is no canonical way to associate a non-formal pseudo-differential operator to a formal one. 
	
	One geometric description of this correspondence between non-formal and formal pseudo-differential operators is the aim of this work. For this, and in order  to circumvent a lack of differential geometric framework, we use diffeologies, a  generalized setting for differential geometry first describe by Chen and Souriau. in this setting, we describe a structure analog to the one of principal bundles, with mild weaker conditions that are described. This leads to open questions that may help to clarify, for example, the results on the index of pseudo-differential operators that actualy obtined through arguments coming from K-theory.   
	\section{On diffeological bundles}
	\subsection{Preliminary notions}
	The main reference for a comprehensive exposition on diffeologies is \cite{Igdiff}. This reference will be completed all along the text for specific concerns. 
	\begin{Definition}[Diffeology] \label{d:diffeology}
		Let $X$ be a set.  A \textbf{parametrisation} of $X$ is a
		map of sets
		$p \colon U \to X$ where $U$ is an open subset of Euclidean space (no fixed dimension).  A \textbf{diffeology} $\p$ on $X$ is a set of
		parametrisations satisfying the following three conditions:
		\begin{enumerate}
			\item (Covering) $\forall x\in X,$  
			$\forall n \in \N$, the constant function $p\colon \R^n\to\{x\}\subset X$ is in
			$\p$.
			\item (Locality) Let $p\colon U\to X$ be a parametrisation such that for
			every $u\in U$ there exists an open neighbourhood $V\subset U$ of $u$
			satisfying $p|_V\in\p$. Then $p\in\p$.
			\item (Smooth Compatibility) Let $(p\colon U\to X)\in\p$.
			Then for every $n$, every open subset $V\subset\R^n$, and every
			smooth map $F\colon V\to U$, we have $p\circ F\in\p$.
		\end{enumerate}
		A set $X$ equipped with a diffeology $\p$ is called a
		\textbf{diffeological space}, and the parametrisations $p\in\p$ are called \textbf{plots}.
	\end{Definition}
	Let $(X,\p)$ and $(Y,\p')$ be two diffeological space. Then $f:X\rightarrow Y$ is smooth if $f(\p)\subset \p'.$
	At this step of the exposition, and for the clarity of developments above, we have to precise some technical features and necessary notations. We have to precise that any finite or infinite dimensional manifold $M$ has a diffeology $\p_\infty(M)$ called ``nebulae'', constructed as follows:
	{{}
		{Let }$O${ be an open subset of a Euclidian space; } $$\p_\infty(M)_O=
		\coprod_{p\in\N}\{\, f : O \rightarrow M; \,  f \in C^\infty(O,M) \quad \hbox{(in
			the usual sense)}\}$$
		and 
		$$ \p_\infty(\F) = \bigcup_O \p_\infty(\F)_O,$$
		where the latter union is extended over all open sets $O \subset \R^n$ for $n \in \N^*.$ 
	}
	
	We choose to note by $\p(X)$ the diffeology of a diffeological space $X.$ From this first diffeology, one can produce many others, and especially two: 
	\begin{itemize}
		\item the 1-dimensional diffeology $\p_1(X)$ made of plots $p \in \p(X)$ that read locally as $p|_O = \gamma \circ f$ where $\gamma \in C^\infty(\R,X)$ , $O$ is an open subset of an Euclidian space and $f \in C^\infty(O,\R).$
		\item the nebulae diffology, which definition which follows refines the one given independently in \cite{Ma2006-3,Ma2013} and \cite{Wa}, and extends the definition of the nebulae diffeology of a manifold.
		{{}
			{Let }$O${ be an open subset of a Euclidian space equipped with its $\p_1-$diffeology; } $$\p_\infty(\F)_O=
			\coprod_{p\in\N}\{\, f : O \rightarrow X; \, f \subset C^\infty(O,X) \quad \hbox{in
				the sense of } \p_1(O) \hbox{ and } \p(X)\}$$
			and 
			$$ \p_\infty(\F) = \bigcup_O \p_\infty(\F)_O,$$
			where the latter union is extended over all open sets $O \subset \R^n$ for $n \in \N^*.$ This notion refines the notion of nebulae diffeology in the following sense: there exists some diffeological spaces $X$ which have very few smooth maps from $X$ to $\R,$ and the definition of the nebuale diffeology present in e.g. \cite{Ma2006-3,Ma2013} is based on the existence of ``enough'' smooth functions in $C^\infty(X,\R).$ Therefore, this "new" notion of nebulae diffeology seems more intrinsic to us, and not exactly the same as the one presented in the existing litterature.
		}   
	\end{itemize}

	A group $G$ is a {\bf diffeological group} if multiplication and inversion are smooth for the underlying diffeology.
	
	A diffeological group has a kinematic tangent space at identity, but the bracket may not exist. If there is a smooth bracket, $G$ is called {\bf diffeological Lie group}. If moreover there is an exponential map from $C^\infty(\R,\mathfrak{g})$ to $C^\infty(\R,G),$ that is, a smooth map $exp$ such that $exp(v)(t)=g(t)$ if and only if $g$ is the unique solution of the (right logarithmic) differential equation 
	$$ \partial_t g . g^{-1} = v,$$ with initial value $g(0)=1_G,$ then $G$ is called {\bf regular} \cite{Les}. 
	If $\mathfrak{g}$ is regular as an abelian diffeological group (vector space) we say that $G$ is fully regular. All these notions cooincide with more classical notions of regular Lie groups (see e.g. \cite{KM,Neeb2007}) and are reviewed in \cite{MR2019,MR2016}.
\subsection{On the way to diffeological principal bundles}
	We describe here a class of objects that generalize principal bundles. Since then, to our knowledge, most results are stated for diffeological principal bundles \textbf{with local trivializations} while there exists examples where we do not know whether a system of local trivializations exist or not, and certainly examples where such local trivializations do not exist. 
	\begin{Definition}
	Let $G$ be a diffeological Lie group acting, smoothly and freely, on the right on a diffeological (total) space $P.$ Then we get a \textbf{diffeological principal bundle } $P$ with structure group $G$ over $P/G.$ Here, $P/G$ is equipped with any diffeology for which the canonical projection $\pi:P \rightarrow P/G$ is smooth.
\end{Definition}
\begin{rem}
	This definition is even weaker than the definition of ``structure quantique'' by Souriau in \cite{Sou} where the problem of the diffeology of the base was not considered.
\end{rem} 
Classicaly, connections are 1-forms $\theta \in \Omega^1(P,\mathfrak{g})$ (see \cite{Igdiff,Ma2013} for a comprehensive definition of differential forms, de Rham differential and wedge product) that are covariant under the right action of $G,$ that is, denoting the right action of $g \in G$ by $R_g,$ we have: 
$$ (R_g)_* \theta = Ad_{g^{-1}} \theta.$$
We keep this definition for a connection 1-form in the considered class of diffeological principal bundles, mimicking what \cite{Ma2013} did in a more specific context, even if the existence of connection 1-forms remains to be checked on the examples considered (see e.g. \cite{KW2016} for a slightly different framework in which similar technical issues are analyzed). In \cite{Igdiff}, another notion of connection is defined, developed independently in \cite{Ma2013} under the terminology of horizontal lifts, which consists in lifting smooth paths of the base to smooth horizontal paths on the total space. Here again, we have to investigate deeply this relationship, already established in a more restricted framework \cite{Ma2013}.  

		\begin{Theorem} If $G$ is fully regular, then any connection 1-form $\theta$ defines a horizontal path projection, i.e. a smooth map $$H : C^\infty(\R,P) \to C^\infty(\R,P)$$ such that $H\gamma(0) = \gamma(0)$ and $\theta(\partial_t H\gamma)=0,$ and which is $G-$equivariant.
			 
			Conversely, if there exists such a horizontal path projection $H,$ there exists only one connection 1-form $\theta$ which defines $H.$
	\end{Theorem} 
\begin{proof}
	Let $\gamma \in C^\infty(\R,P).$ Consider the principal bundle $\gamma^*P$ over $\R$ and with structure group the fully regular group $G.$ Then the differential equation $$\partial_t g . g^{-1} = -\theta(\partial_t \gamma(t))$$ has a unique solution. We define $H(\gamma)(t) = \gamma(t).g(t)$ which is by construction a smooth map, and a direct computation shows that $$\theta(\partial_tH\gamma)=0$$ and that it is $G-$equivariant.  The property $H\circ H = H$ is obvious. The converse is noticed in \cite{Igdiff}.
\end{proof} 

We have now to investigate the lifts of paths on the base $P/G$ to horizontal paths on the total space $P.$
\begin{Theorem}
	Let $\theta$ be a connexion 1-form on $P.$ 
	If $G$ is fully regular, then there exists a horizontal lift $$L : C^\infty(\R,P/G) \rightarrow C^\infty(\R,P)$$
	such that $$ H = L \circ \pi$$
	 if and only if $$\p_1(P/G) = \pi_* \p_1(P).$$
	 Moreover, $L$ is smooth if and only if $$\p(P/G) = \pi_* \p(P).$$ 
\end{Theorem}

\begin{proof}
	Since $H$ exists, $L$ exists if and only if any smooth path on $P/G$ is locally the projection of a smooth path of $P,$ which can be formulated as $\p_1(P/G)= \pi_*p_1(P).$
	
	Since $H$ is smooth, analysing the canonical family of plots generating the functional diffeologies of $C^\infty(\R,P)$ and $C^\infty(\R,P/G),$ we obtain the condition announced  $\p(P/G) = \pi_* \p(P).$ 
\end{proof}
Before stating next theorem, we need to weaken the notion of star-shaped set. 
\begin{Definition}
	Let $X$ be a diffeological spaces. Let $O \subset X$ be a non-empty, path connected set and let $x_0 \in O.$ the set $O$ is \textbf{weakly star-shaped} (for its subset diffeology) if there exists a smooth map $$ Rad_{x_0}: O \rightarrow C^\infty([0,1],O)$$ such that 
	\begin{itemize}
		\item $\forall x \in O, Rad_{x_0}(x)(0) = x_0$
		\item $\forall x \in O, Rad_{x_0}(x)(1) = x.$
	\end{itemize}
\end{Definition}
\begin{rem}
Trivially, a star-shaped set is weakly star-shaped
\end{rem} 
\begin{rem} By its smoothness, the map $Rad_{x_0}$ defined a smooth contraction from $O$ to $\{x_0\},$ which shows that a weakly star-shaped set needs to be smoothly contractible.\end{rem}
 \begin{Theorem} \label{local slice}
	Let $\theta$ be a connection on $P.$ 
	If:
	\begin{itemize}
		\item $G$ is fully regular,
		\item $L$ is smooth,
		\item the base $P/G$ can be covered by open subsets that are weakly star-shaped, 
	\end{itemize}
	then $P$ is locally trivial.
\end{Theorem}
\begin{proof}
	If $O \subset P/G$ is open and weakly star-shaped, then there exists a basepoint $x_0 \in O$ and ``radial'' paths starting at $x_0$ and covering $O,$ that is, a smooth map 
	$$Rad : O \rightarrow C^\infty([0,1],O)$$
	defined as before.
	Therefore, the map
	$$\Phi : (x,g) \in O \times G \mapsto \left(L(Rad_x)(1)\right) .g$$ furnishes the desired local trivializations. 
\end{proof}
	\section{Application to groups of pseudo-differential operators}
	\subsection{Preliminaries on classical pseudodifferential operators}
	We introduce groups and algebras of non-formal pseudodifferential operators needed to set up our equations. 
	Basic definitions are valid for real or complex finite-dimensional vector bundles $E$ 
	over a compact manifold $M$ without boundary whose typical fiber is a finite-dimensional 
	{real or complex} vector space $V$. 
	We begin with the following definition after \cite[Section 2.1]{BGV}.
	
	\begin{Definition} 
		The graded algebra of differential operators acting on the space of 
		smooth sections $C^\infty(M,E)$ is the algebra $DO(E)$ generated 
		by:
		
		$\bullet$ Elements of $End(E),$ the group of smooth maps $E \rightarrow 
		E$ leaving each fibre globally invariant and which restrict to linear 
		maps on each fibre. This group acts on sections of $E$ via (matrix) 
		multiplication;
		
		$\bullet$ The differentiation operators
		$$\nabla_X : g \in C^\infty(M,E) \mapsto \nabla_X g$$ where $\nabla$ 
		is a connection on $E$ and $X$ is a vector field on $M$.
	\end{Definition}
	
	Multiplication operators are operators of order $0$; differentiation 
	operators and vector fields are operators of order 1. In local 
	coordinates, a differential operator of order $k$ has the form
	$ P(u)(x) = \sum p_{i_1 \cdots i_r}(x) \nabla_{x_{i_1}} \cdots 
	\nabla_{x_{i_r}} u(x) \; , \quad r \leq k \; ,$
	{ in which $u$ is a (local) section} and the coefficients $p_{i_1 \cdots i_r}$ can be matrix-valued.
	The algebra $DO(M,E)$ is filtered by order: we note by $DO^k(M,E)$,$k \geq 0$, the differential operators of 
	order less or equal than $k$.  
	
	Now we embed $DO(M,E)$ into the algebra of classical 
	pseudodifferential operators. We need to assume that the reader is familiar with the basic facts on 
	pseudodifferential operators 
	defined on a vector bundle $E \rightarrow M$; these facts can be found for instance 
	in \cite{Gil}, in the review \cite[Section 3.3]{Pay2014}, and in the papers \cite{BK} and \cite{Wid} in which
	the authors construct a global symbolic calculus for pseudodifferential operators showing, for instance, how the 
	geometry of the base manifold $M$ furnishes an obstruction to generalizing 
	local formulas of composition and inversion of symbols.
	
	\vskip 6pt
	\noindent
	\textbf{Notations.} 
	We note by  $ PDO (M,E) $ the space of pseudodifferential operators on smooth sections of $E$, see 
	\cite[p. 91]{Pay2014}; by $ PDO^o (M,E)$ the space of pseudodifferential operators of order $o$; and by 
	$Cl(M,E)$ the space of classical pseudodifferential operators acting on 
	smooth sections of $E$, see \cite[pp. 89-91]{Pay2014}. 
	We also note by $Cl^o(M,E)= PDO^o(M,E) \cap Cl(M,E)$ the space of classical 
	pseudodifferential operators of order $o$, and by 
	$Cl^{\ast}(M,E)$ the group of units of $Cl(M,E)$. 
	
	\vskip 6pt
	
	A topology on spaces of classical pseudodifferential operators has 
	been described implicitely by Kontsevich and Vishik in \cite{KV1}: it is a Fr\'echet topology (and therefore
	it equips $Cl(M,E)$ with a smooth structure) such that each space 
	$Cl^o(M,E)$ is closed in $Cl(M,E).$ This topology is discussed in 
	\cite[pp. 92-93]{Pay2014}, see also  \cite{CDMP,PayBook,Scott} for more explicit and rigorous descriptions.
	We will refer to it in this work under the terminology of \textbf{Kontsevich-Vishik topology}. 
	
	We set
	$$ PDO^{-\infty}(M,E) = \bigcap_{o \in \Z} PDO^o(M,E) \; .$$
	It is well-known that $PDO^{-\infty}(M,E)$ is a two-sided ideal 
	of $Cl(M,E)$, closed for the Kontsevich-Vishik topology, see \cite{Gil} and also  \cite{Scott} for topological aspects. 
	This fact allows us to define the quotients
	$$\F Cl(M,E) = Cl(M,E) /  PDO^{-\infty}(M,E)\; ,$$
	and
	$$ \quad \F Cl^o(M,E) = Cl^o(M,E)  / PDO^{-\infty}(M,E)\; .$$
	The script font $\F$ stands for {\it  formal } pseudodifferential operators. The quotient $\mathcal{F}PDO(M,E)$ 
	is an algebra isomorphic to the space of formal symbols, see \cite{BK}, 
	and the identification is a morphism of $\mathbb{C}$-algebras for 
	the usual multiplication on formal symbols (appearing for instance in \cite[Lemma 1.2.3]{Gil} and
	\cite[p. 89]{Pay2014},  and in \cite[Section 1.5.2, Equation (1.5.2.3)]{Scott} for the particular case of
	classical symbols). Sets of formal classical pseudodifferential operators have been equipped with a topology, before and independently of the works of Kontsevich and Vishik, in \cite{ARS}. 
	
	
	\begin{Theorem} \label{2} 
		The groups $Cl^{0,*}(M,E)$ and $\mathcal{F}Cl^{0,*}(M,E)$, in which  
		${\mathcal F}Cl^{0,*}(M,E)$ is the group of units of the algebra 
		${\mathcal F}Cl^{0}(M,E)$, are regular Fr\'echet Lie groups equipped with smooth exponential maps. 
		Their Lie algebras are $Cl^{0}(M,E)$ and ${\mathcal F}Cl^{0}(M,E)$ respectively.
	\end{Theorem}
	
	Regularity is reviewed in \cite{MR2016} and also in Paycha's lectures,
	see \cite[p. 95]{Pay2014}. The Lie  group structure of $Cl^{0,*}(M,E)$ is discussed in 
	\cite[Proposition 4]{Pay2014}. Theorem \ref{2} is essentially proven in \cite{Ma2006}: it is noted in
	this reference that the results of \cite{Gl2002} imply that the 
	group $Cl^{0,*}(M,E)$ (resp. $\F Cl^{0,*}(M,E)\,$) is open in 
	$Cl^0(M,E)$ (resp. $\F Cl^{0}(M,E)\,$) and that therefore it is a regular 
	Fr\'echet Lie group. 
	\subsection{Diffeologies, topologies and quotients by smoothing operators}
		We equip the ideal of smoothing operators $Cl^{-\infty}$ with its Fr\'echet topology on smooth kernels, which can be understood as the subset topology for the Kontsevich-Vishik topology. The algebras $PDO(M,E),$ $Cl(M,E)$ and $\F Cl(M,E)$ are equipped with diffeologies such that: 
		\begin{enumerate}
			\item diffeologies make addition, multiplication and inversion smooth (i.e. they are diffeological algebras), and the $L^2-$adjoint operation $(.)^*$ is smooth.
			\item The subset diffeology on $PDO^{-\infty}(M,E)$ is the nebulae diffeology on smooth kernels,
			\item the quotient maps  $PDO(M,E)\rightarrow \F PDO(M,E)$ and $Cl(M,E) \rightarrow \F Cl(M,E)$ are smooth.
		\end{enumerate}, .  
	We have to notice that 
	\begin{itemize}
		\item if $Cl(M,E)$ equipped with the Kontsevich-Vishik topology,
		\item if $\F Cl(M,E)$ is equipped with the Adams, Ratiu and Schmid topology,  
	\end{itemize}
then their nebulae diffeology fulfill these assumptions. But we wish to avoid non necessary restrictions of frameworks in this work. Indeed, these topologies are actually adapted to spectral and microlocal analysis in the actual state of the art, and have never been questionned to our knowledge.  This is not the place here to make this investigation, but in order to leave this question open, we only choose to assume assumptions (1-3), keeping in mind the two topologies mentionned before (and their underlying nebulae diffeologies) as a field of application.

	 Let $Cl^*(M,E)$ (resp. $\F Cl^*(M,E)$) the group of invertible operators of $Cl(M,E)$ (resp. $\F Cl(M,E)$), equipped with their subset diffeology in $Cl$ (resp. $\F Cl(M,E)$). Notice that the Lie brackets exist in these algebras.

	\begin{Proposition} \label{qG}
		\item  If $$G = \{Id + R \in Cl^* \, | \, R \in PDO^{-\infty}\},$$ then
		$\F Cl^* = Cl^* / G$ algebraically
		\end{Proposition}
	\begin{proof}
		We have directly that $Cl^*(M,E)/G \subset \F Cl^*(M,E).$Now, let $A \in Cl(M,E) $ such that its associated formal operator $a$ is invertible in $\F Cl(M,E)$.Then $a^{-1}$ is the formal operator associated to some operators $B \in Cl(M,E) $ (resp. $PDO(M,E)$) such that $AB = Id + R,$ where $R \in PDO^{-\infty}(M,E).$
		
		Let us first assume that $A$ (and $a$) are of order $0.$ Then the partial symbol of order $0$ (the principal symbol) is invertible and hence the operator $A$ is Fredholm. Setting $p_K$ the $L^2$ orthogonal projection on its kernel, and $p_I$ the orthogonal projection on the orthogonal complement of the image of $A.$ These two subspaces are finite dimensional and made of smooth sections. Therefore, one can find $(\lambda, \mu) \in \R^2$ such that $$A' = A + \lambda P_I + \mu P_K $$
		and such that $A' \in Cl^{0,*}(M,E).$ By construction, the formal symbol associated to $A'$ is $a,$ and we get \begin{eqnarray*} R &=& (A')^{-1} A - Id \\ &= &(A')^{-1} (A'-\lambda P_I - \mu P_K)  - Id \\ &=& (A')^{-1} (-\lambda P_I - \mu P_K) \in PDO^{-\infty}.(M,E).\end{eqnarray*}
		
		Now, of $ord(A) = o \in \Z,$ let $\Delta$ be a (positive) Laplacian on $C^\infty(M,E)$ then $(Id + \Delta)^{-o/2} A$ is of order $0,$ to which one can apply the last construction. We finish the proof by remarking that smoothing operators are not invertible in $Cl(M,E).$
	\end{proof}
	
\subsection{Diffeological short exact sequences and principal bundles with connections}
The construction of $\F Cl(M,E)$ can be rephrased into a short exact sequence
	$$0 \rightarrow PDO^{-\infty}(M,E) \rightarrow Cl(M,E) \rightarrow \F Cl(M,E) \rightarrow 0 $$
	which is \textit{diffeological}, that is, each arrow is smooth in the diffeological sense according to assumptions (1-3). According to Proposition \ref{qG} and under the same diffeological conditions, 
	we get a short diffeological exact sequence of groups
	\begin{equation} \label{eq} 1 \rightarrow G \rightarrow Cl^*(M,E) \rightarrow \F Cl^*(M,E) \rightarrow 1. \end{equation}
	We get analogous properties for diffeological exact sequences of bounded operators 
	$$0 \rightarrow PDO^{-\infty}(M,E) \rightarrow Cl^0(M,E) \rightarrow \F Cl^0(M,E) \rightarrow 0 $$
	$$\hbox{and }1 \rightarrow G \rightarrow Cl^{0,*}(M,E) \rightarrow \F Cl^{0,*}(M,E) \rightarrow 1.$$
	These exact sequences
	have actually no local slice $\F Cl^{*}(M,E) \rightarrow Cl^{*}(M,E)$ or $\F Cl^{0,*}(M,E) \rightarrow Cl^{0,*}(M,E),$
	in other words the diffeological $G-$principal bundle $Cl^{0,*}$ over $\F Cl^{0,*}$ has actually no local trivialization (according to the  Adams, Ratiu and Schmid topology which is the only one studied explicitely).  This is at this point that we carry now a new element, that generalize the constructions given in \cite{Ma2021-1} which treats the case $M=S^1.$
	\begin{Theorem} \label{smoothing}
		There exists $PDO^{-\infty}-$valued connections on the principal bundle $Cl^{*}(M,E)$ over  $\F Cl^{*}(M,E).$ In particular,
		 we define, for $s \in Cl^{-\infty}(M,E)$ and $(a,b)\in (Cl(M,E))^2,$ three such classes of connections:
		$$ \Theta^{s,l}_a b = sas^* b,$$
		$$ \Theta^{s,r}_a b = bsas^* ,$$
		$$ \Theta^{s,[]}_a b = [sas^* ,b],$$
		and extend them by right-invariance on the full space $TCl^*(M,E).$
	\end{Theorem}
\begin{proof}
	Recall that, for the diffeologies under consideration, multiplication, adjoint and addition are smooth. Therefore, the right invariant vector fields on $TCl^*(M,E)$ drfined by 
	$$ g \in Cl^*(M,E) \mapsto g^{-1}a g,$$ 
		$$ g \in Cl^*(M,E) \mapsto g^{-1}b g,$$ 
			$$ g \in Cl^*(M,E) \mapsto g^{-1}s g$$
				$$\hbox{ and } g \in Cl^*(M,E) \mapsto g^{-1}s^* g$$ are smooth vector fields, and for the same reasons, when evaluation on right-invariant vector fields above, the map
				$$(s,a,b) \mapsto (\Theta^{s,l}_a b,\Theta^{s,r}_a b ,\Theta^{s,[]}_a b)$$ is smooth.  
				
	By direct computation, each connection form is $G-$ covariant.
\end{proof}
\begin{rem} By restriction, these constructions also show  that there exists $PDO^{-\infty}-$valued connections on the principal bundle $Cl^{0,*}(M,E)$ over $\F Cl^{0,*}(M,E).$ \end{rem}

Therefore, we have the following: 

\begin{rem}
	If there exists diffeologies on  $Cl^*(M,E)$ and $\F Cl^*(M,E)$ such that Theorem \ref{local slice} applies to at least one $PDO^{-\infty}(M,E)-$valued connection, then there exists local trivializations for the diffeological principal bundle $Cl^*(M,E)$ over $\F Cl^*(M,E)$ and hence local slices to the exact sequence (\ref{eq}).
\end{rem}
	
	
	

\subsection{Yet another class of connections and the Schwinger cocycle when $M=S^1$}
	There is another class of smoothing connections defined by 
	$$\Theta^{\epsilon(D)}_a b = b [a,\epsilon(D)]$$
	where $\epsilon(D) = D . |D|^{-1},$ and $D = -i \frac{d}{dx}.$
	This connection fulfills also the requirements of Theorem \ref{smoothing} by the same arguments, see e.g. \cite{Ma2021-1}. 
	If the diffeology on $Cl^*$ has ``enough'' 2-dimensional plots, one can define the curvature of the connection \cite{Ma2013}. 
	
	The same property of the curvature holds for a second class of connections, studied in \cite{Mapreprint}: 
	$$ \Theta^+_a b = \frac{1}{2}b (1 + \epsilon(D))a.$$
	\begin{rem}We have here to notice that $\Theta^+$ is not $PDO^{-\infty}(M,E)-$valued, but its curvature is. \end{rem}
	Then we get, following \cite{Mapreprint} and denoting by $\Omega$ the curvature of $\Theta^+:$
	
	\begin{Theorem}
		$tr(\Omega(a,b)(1+\epsilon))$ is cohomologous to the Schwinger cocycle and $tr(\Omega^2(a,b,c,d)(1+\epsilon))$ has non-vanishing cohomology class.		\end{Theorem}
	
	\section{Open problems and perspectives} 
	The principal bundle construction for $Cl^*$ gives new geometric lights about the correspondence between non-formal and formal pseudo-differential operators. To our knowledge, this correspondence is only treated actually from the (algbraic) viewpoint of K-theory. This approach raises two questions: 
	\begin{enumerate}
		\item the question of the adequate topologies for such groups of operators. this question is difficult due to the hard technicalities that one can meet when dealing with such objects. But the question of local maps from $\F Cl^*(M,E)$ to $Cl^*(M,E)$ has its own interest. One can ask if the statement of homotopy equivalence between $Cl^*(M,E)$ and $^\F Cl^*(M,E)$ can be obtained through a weak contractibility of $G.$ The necessary results actually do not exist in the context of diffeological principal bundles in the class that we have to consider.
		\item This first question is related to index theory, actually again only reated through K-theory, r in a slightly different approach through the restricted linear group following \cite{PS} for $M=S^1.$ This leads to a second question: in our framework, the index cocycles appear as quite natural differential geometric objects, even for unbounded operators.  Indeed, the results by \cite{MW2017} may be applied if the necessary assumptions were fulfilled, but this is actually for us an open question. 
	\end{enumerate}  
	Therefore, a key problem remains in the investigation of the topologies of interest for $Cl(M,E)$ and $\F Cl(M,E)$ and their underlying nebulae diffeologies. We have at hand an example of such topologies, but the key properties for the application of our results still need to be investigated. 

\end{document}